\documentclass[a4paper, 12pt, twoside]{amsart}

\usepackage{url}

\usepackage{tikz-cd}   
\usepackage{amssymb}   
\usepackage{enumerate} 

\theoremstyle{plain}
\newtheorem{theorem}{Theorem}[section]
\newtheorem*{theorem*}{Theorem}
\newtheorem{lemma}[theorem]{Lemma}
\newtheorem{proposition}[theorem]{Proposition}

\theoremstyle{definition}
\newtheorem{definition}[theorem]{Definition}

\newtheorem{example}[theorem]{Example}
\newtheorem{remark}[theorem]{Remark}

\DeclareMathOperator{\Ext}{Ext}
\DeclareMathOperator{\Hom}{Hom}

\DeclareMathOperator{\iso}{\cong}
\DeclareMathOperator{\rk}{rk}
\DeclareMathOperator{\Pic}{Pic}

\DeclareMathOperator{\num}{Num}

\DeclareMathOperator{\im}{im}

\DeclareMathOperator{\Aut}{Aut}
\DeclareMathOperator{\Coh}{Coh}

\DeclareMathOperator{\rad}{rad}

\DeclareMathOperator{\C}{\mathbb{C}}
\DeclareMathOperator{\Z}{\mathbb{Z}}
\DeclareMathOperator{\Q}{\mathbb{Q}}
\DeclareMathOperator{\vp}{\varphi}

\DeclareMathOperator{\SpCoh}{Sp-Coh}
\DeclareMathOperator{\SL}{SL}

\begin{document}

\title{Derived Autoequivalences of Bielliptic Surfaces}

\author{Rory Potter}

\address{School of Mathematics and Statistics, University of
  Sheffield, Hicks Building, Hounsfield road, Sheffield, South
  Yorkshire, S3 7RH}
\email{rdpotter1@sheffield.ac.uk}

\date{\today}

\begin{abstract}
  We describe the group of exact autoequivalences of the bounded
  derived category of coherent sheaves on a bielliptic surface. We
  achieve this by studying its action on the numerical Grothendieck
  group of the surface.
\end{abstract}

\thanks{I am supported by a Graduate Teaching Assistantship by the
  University of Sheffield.}
\maketitle

\section{Introduction}
\label{sec:intro}

Let $X$ be a smooth projective variety over the complex numbers. We
can construct the bounded derived category of coherent sheaves on $X$
denoted by $D(X) = D^b\Coh(X)$. It is natural to study the symmetries
of $D(X)$ which preserve the intrinsic structure: the group
$\Aut D(X)$ of exact $\C$-linear autoequivalences of $D(X)$ considered
up to isomorphism as functors. We think of these autoequivalences as
``higher'' symmetries of the variety. Several autoequivalences of
$D(X)$ arise naturally forming the subgroup
\[
\Aut_{st} D(X) = \left( \Aut X \ltimes \Pic X \right) \times \mathbb{Z}
\]
of \emph{standard autoequivalences} of $\Aut D(X)$. This subgroup is
generated by pulling back along automorphisms of $X$, tensoring by
line bundles and by powers of the shift functor.

It is natural to ask if there are any others? When the
(anti-)canonical bundle of $X$ is ample, Bondal and Orlov
\cite{MR1818984} showed that $\Aut D(X) = \Aut_{st} D(X)$, i.e. there
are no extra autoequivalences of $D(X)$. The first example of a
non-standard autoequivalence was observed by Mukai \cite{MR607081} for
principally polarized abelian varieties. Many people have tried to
understand non-standard autoequivalences of the derived category but
the full group $\Aut D(X)$ is only understood in a small number of
cases. Orlov \cite{MR1921811} computed the full group for Abelian
varieties. Together with Bondal and Orlov's result, this classifies
the group of autoequivalences of the derived category of smooth
projective curves. Broomhead and Ploog \cite{MR3162236} computed the
group for many rational surfaces (including most toric
surfaces). Bayer and Bridgeland \cite{2013arXiv1310.8266B} described
the group for K$3$ surfaces of Picard rank $1$. Uehara
\cite{2015arXiv150106657U} conjectured a description of the group for
smooth projective elliptic surfaces of non-zero Kodaira dimension and
proved the conjecture when each reducible fibre is a cycle of
$(-2)$-curves. Furthermore, he describes the group for elliptic ruled
surfaces \cite{2015arXiv151106031U}. Ishii and Uehara \cite{MR2198807}
computed the group for smooth projective surfaces (not necessarily
minimal) of general type whose canonical model has at worst $A_n$
singularities. No other examples are completely understood at this
time. In this paper we describe the group $\Aut D(S)$ when $S$ is a
bielliptic surface.  \newline

Let $S$ be a bielliptic surface, $\widetilde{S}$ the abelian surface
which is the canonical cover of $S$, and $N(S)$ the numerical
Grothendieck group of $S$. Denote by $O_{\Delta}(N(S))$ the subgroup
of isometries of $N(S)$ which preserve
\[
  \Delta = \left\{ [E] \in N(S) \middle| [E] = \pi_!([\widetilde{E}])
    \text{ for some } [\widetilde{E}] \in N(\widetilde{S}) \right\}
\]
where $\pi_! \colon N(\widetilde{S}) \to N(S)$ is the pushforward on
$K$-theory. The main result is the following.

\begin{theorem}\label{thm:main}
  There is an exact sequence
  \begin{equation*}
    \begin{tikzcd}[column sep=small]
      1 \arrow{r} & (\Aut S \ltimes \Pic^0 S) \times \mathbb{Z}
      \arrow{r} & \Aut D(S) \arrow{r}{\rho} & O_{\Delta}(N(S))
    \end{tikzcd}
  \end{equation*}
  where $\mathbb{Z}$ is generated by the second shift $[2]$. The map
  $\rho$ is induced by the natural action of $\Aut D(S)$ on $N(S)$
  given by $\rho(\Phi)[E] = [\Phi(E)]$. Furthermore, the image of
  $\rho$ is a subgroup of $O_{\Delta}(N(S))$ of index $4$ if $S$ is of
  type $2$ or $4$ and index $2$ otherwise (see Table
  1).
\end{theorem}

Moreover, we describe the generators of $\Aut D(S)$ in some cases.

\begin{theorem}\label{thm:generators}
  Suppose $S$ is a split bielliptic surface (see Definition
  \ref{def:split}). Then the group $\Aut D(S)$ is generated by
  standard autoequivalences and relative Fourier-Mukai transforms
  along the two elliptic fibrations.
\end{theorem}

The article is structured as follows. In section
\ref{sec:preliminaries} we review preliminary material on bielliptic
surfaces, the numerical Grothendieck group, canonical covers and
relative Fourier-Mukai transforms. In section \ref{sec:proof} we prove
Theorem \ref{thm:main} by describing a collection of autoequivalences
arising from moduli spaces of stable, special sheaves whose Chern
character lies in $\Delta$. In section
\ref{sec:relat-four-mukai-bielliptic} we prove Theorem
\ref{thm:generators}.

\subsection*{Acknowledgments}
\label{sec:aclnowledgements}

I want to thank my supervisor Tom Bridgeland for suggesting this
problem and for his guidance, patience and support. I also want to
thank Evgeny Shinder for many helpful conversations and for alerting
me to an error in an earlier version of this paper. I also want to
thank Paul Johnson for a helpful conversation concerning Section
\ref{sec:relat-four-mukai-bielliptic}.

\section{Preliminaries}
\label{sec:preliminaries}

All varieties will be over the complex numbers.

\subsection{Bielliptic surfaces}
\label{sec:bielliptic-surfaces}

\begin{definition}
  A \emph{bielliptic} (or \emph{hyperelliptic}) surface $S$ is a
  minimal projective surface of Kodaira dimension zero with $q=1$ and
  $p_g=0$.
\end{definition}

Bielliptic surfaces are constructed by taking the quotient of the
product of two elliptic curves $A \times B$ by a finite subgroup $G$
of $A$ acting on $A$ by translations and on $B$ by automorphisms, not
all translations. These surfaces are classified by Bagnera and De
Franchis into seven families \cite[\S V.5]{MR2030225} determined by
the subgroup $G$ and the lattice $\Gamma$ such that $B = \C / \Gamma$
(see Table 1).

\begin{table}[!h]
\centering
  \begin{tabular}{ c | c | c | p{5cm}  }
    Type & $\Gamma$ & G & Action of $G$ on $B$ \\ \hline
    1 & Arbitrary & $\Z/2$ & $b \mapsto -b$\\ 
    2 & Arbitrary & $\Z/2 \oplus \Z/2 $ & $b \mapsto -b$, \newline $b \mapsto b + \beta$, where $2\beta = 0$\\ 
    3 & $\Z \oplus \Z \omega$ & $\Z/3 $ & $b \mapsto \omega b$ \\
    4 & $\Z \oplus \Z \omega$ & $\Z/3 \oplus \Z/3$  & $b \mapsto \omega b$, \newline $b \mapsto b + \beta$, where $\omega \beta = \beta$ \\
    5 & $\Z \oplus \Z i$ & $\Z/4$ & $b \mapsto ib$ \\ 
    6 & $\Z \oplus \Z i$ & $\Z/4 \oplus \Z /2$ & $b \mapsto ib$,\newline $b \mapsto b + \beta$, where $i\beta = \beta$ \\
    7 & $\Z \oplus \Z \omega$ & $\Z/6$ & $b \mapsto - \omega b$ \\ \hline
  \end{tabular}
\label{table:types-of-bielliptic-surfaces}
\caption{($\omega^3 = 1$ and $i^4 = 1$ are complex roots of unity.)}
\end{table}

\begin{definition}\label{def:split}
  We call a bielliptic surface \emph{split} if it is of type $1,3,5$,
  or $7$ and \emph{non-split} otherwise.
\end{definition}

\begin{remark}
  Associated to a bielliptic surface $S$ are two elliptic fibrations:
  \begin{align*}
    p_A \colon& S \to A/G\\
    p_B \colon& S \to B/G
  \end{align*}
  with $A/G$ an elliptic curve and $B/G \iso \mathbb{P}^1$.

  Since the projection $A \to A/G$ is \'{e}tale, all the fibres of
  $p_A$ are smooth.  The fibre of $p_B$ over a point $P \in B/G$ is a
  multiple of a smooth elliptic curve. The multiplicity of the fibre
  of $p_B$ at $P$ is the same as the multiplicity of the projection
  $B \to B /G \iso \mathbb{P}^1$. As all smooth fibres of $p_A$
  (respectively $p_B$) are isomorphic to $B$ (respectively $A$) we
  will denote the class of the smooth fibre of $p_A$ and $p_B$ in
  $H^2(S,\Q)$ by $B$ and $A$ respectively.
\end{remark}

The derived category of a bielliptic surface $S$ is a strong invariant
of the surface due to the following result of Bridgeland and Maciocia.

\begin{proposition}[{\cite[Proposition
    6.2]{MR1827500}}]\label{prop:beilliptic-FM-partners}
Let $S$ be a bielliptic surface and $S'$ a smooth projective surface such that 
\[
D(S') \iso D(S).
\]
Then $S$ is isomorphic to $S'$.
\end{proposition}

\subsection{Numerical Grothendieck Group}
\label{sec:numer-groth-group}

The \emph{Grothendieck group} \linebreak $K(X)$ of a smooth projective
variety $X$ is the free group generated by isomorphism classes of objects
in $D(X)$ modulo an equivalence relation given by distinguished
triangles \cite[\S 5]{MR2244106}. There is a natural bilinear form on
this group, the Euler form, defined by
\[
\chi([E],[F]) = \sum_{i \in \mathbb{Z}} (-1)^i \dim_{\mathbb{C}}\Hom_{D(X)}^i(E,F)
\]
where $\Hom^i_{D(X)}(E,F) = \Hom_{D(X)}(E,F[i])$. This bilinear form
is well defined as the Euler form is additive on distinguished
triangles. We can consider the radical of the Euler form
\[
  \rad \chi = \left\{ v \in K(X) \middle| \chi(v,w) =0
    \text{ and } \chi(w,v)= 0 \text{ for all } w \in K(X) \right\}
\]
and form the quotient $N(X) = K(X)/\rad \chi$, which we call the
\linebreak \emph{numerical Grothendieck group} of $X$. The Euler form
descends to a non-degenerate bilinear form on $N(X)$. Recall that
$\num(S)$ is the group of divisors on $S$ modulo numerical equivalence
$\equiv$.

\begin{proposition}\label{prop:num_groth_bielliptic}
  Let $S$ be a bielliptic surface. The Chern character
\[
ch \colon K(S) \to H^{2*}(S, \Q)
\]
identifies $N(S)$ with the group
\[
  H^0(S,\Z) \oplus \num(S) \oplus H^4(S,\Z) \iso \Z \oplus
  \num(S) \oplus \Z.
\]
Under this identification, for $ch(E) = (r,D,s)$ and
$ch(F) = (r',D',s')$ the Euler form becomes
$\chi(E,F) = rs' + r's - D \cdot D'$.
\end{proposition}
\begin{proof}

  For $v = (v_0,v_2,v_4) \in H^{2*}(S, \Q)$ define
  $v^{\vee} = (v_0, -v_2, v_4) \in H^{2*}(S,\Q)$. Recall that the
  Mukai pairing on $H^{2*}(S,\Q)$ is defined by
\[
\langle v, v'\rangle = \int_X v^{\vee} \cdot v'
\]
where the product in the integral is the cup product of cohomology
classes. As the Todd classes of abelian and bielliptic surfaces are
$(1,0,0)$, by Hirzebruch-Riemann-Roch for $[E],[F] \in K(S)$
\[
  \chi([E],[F]) = \langle ch(E),ch(F) \rangle.
\]
Thus the Euler form for $ch(E) = (r,D,s)$ and
$ch(F) = (r',D',s')$ can be written as
\[
  \chi([E],[F]) = \langle (r,D,s), (r',D',s') \rangle = rs' + r's - D
  \cdot D'.
\]
A class lies in the radical of the Euler form if and only if it lies
in the radical of the Mukai pairing. As the Mukai pairing is
non-degenerate an element of $K(S)$ lies in the radical of the Euler
form if and only if it has zero Chern Character. Hence
$\ker(ch) = \rad \chi$ and $\im(ch) \iso N(S)$.

Using this alternative description of the Euler form, we see that
the class of a numerically trivial divisor $D$, $[\mathcal{O}_S(D)]$
is equivalent to $[\mathcal{O}_S]$. Therefore, the image of the Chern
character intersected with the group $H^2(S,\Q)$ is the group
$\num(S)$. Furthermore, by
Hirzebruch-Riemann-Roch we have $ch_2(E) = \chi(E) \in \Z$ for all
$E$. Thus we have an isomorphism
\[ 
N(S) \iso H^0(X,\Z) \oplus \num(S) \oplus H^4(X,\Z) \iso \Z \oplus
\num(S) \oplus \Z.
\]
\end{proof}

\begin{remark}
  Let $A$ be an Abelian surface. Then a similar argument to
  Proposition \ref{prop:num_groth_bielliptic} shows that the Chern
  character induces an isomorphism
  $N(A) \iso \Z \oplus \num(A) \oplus \Z$.
\end{remark}

\begin{remark}
We will study the group $\Aut D(S)$ by studying its
action on the numerical Grothendieck group given by the homomorphism
\begin{equation*}
  \rho \colon \Aut D(S) \to \Aut(N(S))
\end{equation*}
where $\rho(\Phi)([E]) = [\Phi(E)]$. Autoequivalences of $D(S)$
preserves the $\Hom^i$ groups, thus the Euler form. Hence the image of
$\rho$ is contained in the group of isometries $O(N(S))$ of $N(S)$.
\end{remark}

 \subsection{Canonical covers of Bielliptic surfaces}
\label{sec:can-cover-bielliptic}

\begin{proposition}[{\cite[\S 2]{1998math.....11101B}, \cite[\S 7.3]{MR2244106},\cite[\S 7.2]{MR2511017}}]
  Let $X$ be a smooth projective variety whose canonical bundle
  $\omega_X$ has finite order, i.e. there exists $n$ such that
  $\omega_X^{\otimes n} \iso \mathcal{O}_X$. Then there exists a
  smooth projective variety $\widetilde{X}$ with trivial canonical
  bundle, and an \'{e}tale cover $\pi \colon \widetilde{X} \to X$ of
  degree $n$ such that
\[
  \pi_*(\mathcal{O}_{\widetilde{X}}) \iso \bigoplus_{i=0}^{n-1} \omega_X^{\otimes i}.
\]
Furthermore, $\widetilde{X}$ is uniquely defined up to isomorphism,
and there is a free action of the cyclic group
$\widetilde{G} = \Z/n\Z$ on $\widetilde{X}$ such that
$\pi \colon \widetilde{X} \to X = \widetilde{X}/\widetilde{G}$ is the
quotient morphism.
\end{proposition}

The canonical cover of a bielliptic surface will play an important
role in determining the group of autoequivalences. We list the
following facts about the canonical cover of a bielliptic surface and
leave the verification to the reader.

\begin{proposition}
  Let $S$ be a bielliptic surface which is realized as the quotient of
  $A \times B$ be a finite group $G$ of order $nk$. Then there exists
  an abelian surface $\widetilde{S}$ which is the canonical cover of
  $S$. 
\begin{itemize}
\item If $S$ is split, then $k=1$, $\widetilde{S} \iso A \times B$ and
  $G \iso \widetilde{G}$.

\item If $S$ is non-split, then $k>1$ and $\widetilde{S}$ can be
  realized as the quotient $\widetilde{S} \iso (A \times B)/H$ where
  $H$ is the cyclic subgroup of $G$ of order $k$ acting on
  $A \times B$ purely by translations. We have
  $G \iso \Z/n\Z \oplus \Z/k\Z$.
\end{itemize}
\end{proposition}

\begin{remark}\label{rem:can-cover-bielliptic}
The canonical cover $\widetilde{S}$ has two fibrations
\begin{align*}
  \widetilde{p}_A &\colon \widetilde{S} \to A/H\\
  \widetilde{p}_B &\colon \widetilde{S} \to B/H.
\end{align*}

Both $\widetilde{p}_A$ and $\widetilde{p}_B$ are smooth fibrations
with fibres isomorphic to $B$ and $A$ respectively. We will denote the
class of these fibres by $\widetilde{B}$ and $\widetilde{A}$ in
$\num(\widetilde{S})$ respectively. The degree of the intersection
$\widetilde{B} \cdot \widetilde{A} = k = |H|$.
\end{remark}

We summarize the description of $\num(S)$  given by Serrano \cite[\S
1]{MR1038716} in the following lemma.

\begin{lemma}\label{lem:DesOfNum}
  Let $|G| = nk$ and $n = \deg \pi$ where
  $\pi \colon \widetilde{S} \to S$ is the canonical cover of $S$.
\begin{enumerate}
\item The second rational cohomology group $H^2(S,\Q)$ is generated by
  $A$ and $B$.
\item Suppose $S$ is split. Then $k=1$ and the group $\num(S)$ is
  generated by $\frac{1}{n}A$ and $B$.
\item Suppose $S$ is non-split. Then the group $\num(S)$ is generated
  by $\frac{1}{n}A$ and $\frac{1}{k}B$.
\end{enumerate}
\end{lemma}

Consider the category $\SpCoh(S)$ of coherent
$\pi_*(\mathcal{O}_S)$-modules on $S$. A sheaf $E$ lies in $\SpCoh(S)$
if and only if $E \otimes \omega_S \iso E$. We call such sheaves
\emph{special}. The following results from \cite[\S
2]{1998math.....11101B},\cite[\S 7.2]{MR2511017} relate this category
to the category of coherent sheaves on $\widetilde{S}$.

\begin{lemma}
  The functor
\[
\pi_* \colon \Coh(\widetilde{S}) \to \SpCoh(S)
\]

is an equivalence.
\end{lemma}

This descends to the level of derived categories in the following way:
\begin{proposition}\label{prop:derived-sheav-classes}
  Let $E$ be an object of $D(S)$. Then there is an object
  $\widetilde{E}$ of $D(\widetilde{S})$ such that
  $\mathbf{R}\pi_*(\widetilde{E}) \iso E$ if and only if
  $E \otimes \omega_S \iso E$.
\end{proposition}

\begin{remark}
  Recall $\pi_! \colon N(\widetilde{S}) \to N(S)$ is defined by
  (\cite[\S 5.2]{MR2244106})
  \[
    \pi_![E] = \sum_{i \in \Z} (-1)^i[R^i\pi_*(E)].
  \]
  After taking Chern characters, $\pi_!$ coincides with the
  pushforward $\pi_*$ on cohomology by Grothendieck-Riemann-Roch. This
  is due to the Todd classes of $\widetilde{S}$ and $S$ being
  $(1,0,0)$.
\end{remark}

On the level of the numerical Grothendieck group $N(S)$ consider
the subgroup $\Delta$ of \emph{special classes}
\[
  \Delta = \im(\pi_!) = \left\{ [E] \in N(S) \middle| [E] =
    \pi_!([\widetilde{E})]) \text{ for some } [\widetilde{E}] \in
    N(\widetilde{S}) \right\}.
\]

\begin{remark}
  The class $[E]$ of a special object $E \in D(S)$ lies in $\Delta$ by
  Proposition \ref{prop:derived-sheav-classes} as there exists
  $\widetilde{E} \in D(\widetilde{S})$ such that
  $[E] = [\pi_*(\widetilde{E})] = \pi_![\widetilde{E}]$.
\end{remark}

The subgroup $\Delta$ is important because the image of
autoequivalences of $D(S)$ under $\rho$ preserves $\Delta$.

\begin{proposition}\label{prop:auto-delta}
  Let $\Phi \in \Aut D(S)$. Then $\rho(\Phi)$ preserves $\Delta$.
\end{proposition}
\begin{proof}
Any autoequivalence $\Phi \in \Aut D(S)$ lifts to an equivariant
autoequivalence $\widetilde{\Phi} \in \Aut D(\widetilde{S})$ by
\cite[Theorem 4.5]{MR1629929} or \cite[Theorem 7.13]{MR2511017} such
that
\[
  \mathbf{R}\pi_* \circ \widetilde{\Phi} \iso \Phi \circ
  \mathbf{R}\pi_*.
\] 
Consider $v \in \Delta$ and $w \in N(\widetilde{S})$ such that
$v = \pi_!(w)$. Then
\[
  \rho(\Phi)(v) = \rho(\Phi)(\pi_!(w)) =
  \pi_!(\rho(\widetilde{\Phi})(w)) \in \Delta
\]
Therefore $\rho(\Phi)(\Delta) \subset \Delta$.
\end{proof}

\subsection{Relative Fourier-Mukai Transforms}
\label{sec:relat-four-mukai}

Recall that a relatively minimal elliptic surface is a projective
surface $X$ together with a fibration $\pi \colon X \to C$ with
generic fibre isomorphic to an elliptic curve and with no
$(-1)$-curves in the fibres. We will only consider relatively minimal
elliptic surfaces.

For an elliptic surface $\pi \colon X \to C$ define $\lambda_{\pi}$ to
be the smallest positive integer such that $\pi$ has a holomorphic
$\lambda_{\pi}$-multisection. This is equivalent to 
\[ 
  \lambda_{\pi} = \min \{f \cdot D >0 | D \in \num(X)\}.
\]
where $f$ is the class of a smooth fibre of $\pi$.

Suppose $a >0, b \in \Z$ with $gcd(a \lambda_{\pi},b)=1$.  Then we can
construct the moduli space $J_X(a,b)$ of pure dimension 1 stable
sheaves of class $(a,b)$ supported on a smooth fibre of $\pi$.
Bridgeland constructed equivalences between the derived category of
$X$ and the derived category of $J_X(a,b)$ \cite{MR1629929}. We call
these equivalences \emph{relative Fourier-Mukai transforms}.

\begin{theorem}\label{thm:relative-fm}{\cite[Theorem
    5.3]{MR1629929}} Let $\pi \colon X \to C$ be an elliptic surface
  and take an element
 \[ 
   \begin{pmatrix}
     c & a\\
     d & b\\
   \end{pmatrix} 
   \in \SL_2(\Z)
\] 
such that $\lambda_X$ divides $d$ and $a>0$. Let $Y$ be the elliptic
surface $J_X(a,b)$ over $C$. Then there exists sheaves $\mathcal{P}$
on $X \times Y$, flat and strongly simple over both factors such that
for any point $(x,y) \in X \times Y$, $\mathcal{P}_y$ has Chern class
$(0,af,b)$ on $X$ and $\mathcal{P}_x$ has Chern class $(0,af,c)$ on
$Y$. For any such sheaf $\mathcal{P}$, the resulting functor
$\Phi = \Phi^{\mathcal{P}}_{Y \to X} \colon D(Y) \to D(X)$ is an
equivalence and satisfies
\begin{equation*}
\begin{pmatrix}
  r(\Phi(E))\\
  d(\Phi(E))\\
  \end{pmatrix}
=
\begin{pmatrix}
  c & a\\
  d & b\\
\end{pmatrix}
\begin{pmatrix}
  r(E)\\
  d(E)
\end{pmatrix}
\end{equation*}
for all objects $E$ of $D(Y)$.  
\end{theorem}

\section{Proof of Theorem \ref{thm:main}}
\label{sec:proof}

To prove Theorem \ref{thm:main} we will compute the kernel and image of
\[
  \rho \colon \Aut D(S) \to O(N(S))
\] 
given by $\rho(\Phi)([E])= [\Phi(E)]$.

We will need the following result concerning moduli spaces of sheaves
on a bielliptic surface $S$ which will give rise to autoequivalences
of the derived category.

\begin{proposition}\label{prop:moduli}
  Let $S$ be a bielliptic surface and $\pi \colon \widetilde{S} \to S$
  the canonical cover of $S$. Take $v \in \Delta$ such that $v$ is
  isotropic ($\chi(v,v) = 0$) and there exists $v' \in N(S)$ such that
  $\chi(v,v')=1$. Choose a generic ample line bundle $H$ with respect
  to $v$. Then there exists a two dimensional, projective, smooth,
  fine moduli space $M$ of $H$-slope stable, special sheaves on $S$ of
  class $v$.

  Moreover, the universal sheaf on $M \times S$ induces an
  autoequivalence $\Phi$ of $D(S)$ such that
  $[\Phi(\mathcal{O}_s)] = v$ for any closed point $s \in S$.

  \begin{proof}
    Choose a generic ample divisor $H$ which does not lie on a wall
    with respect to $v$.
 
    First we show that $M$ is non-empty. As $v \in \Delta$, there
    exists $w \in N(\widetilde{S})$ such that $\pi_!(w)=v$. The moduli
    space of $\pi^*H$-Gieseker semistable sheaves of class $w$ on the
    abelian surface $\widetilde{S}$ is non-empty by \cite[\S
    4.3]{MR2665168}. Let $F$ be a $\pi^*H$-Gieseker semistable sheaf
    of class $w$. As $\pi^*H$-Gieseker semistable sheaves are
    $\pi^*H$-slope semistable, $F$ is $\pi^*H$-slope semistable. By
    \cite[Proposition 1.5]{MR0337967} the pushforward $\pi_*F$ is a
    $H$-slope semistable sheaf as $\pi$ is finite \'{e}tale. By
    construction, $[\pi_*F] = \pi_!(w) = v$.  Therefore, the moduli
    space $M_H^{ss}$ of $H$-slope semistable sheaves of class $v$ is
    non-empty.

    As $H$ was chosen not to lie on a wall and there exists $v'$ such
    that $\chi(v,v') =1 $, all $H$-slope semistable sheaves are
    $H$-slope stable. Therefore the moduli space $M_H$ of $H$-slope
    stable sheaves is projective. By \cite[Proposition 4.6]{MR2665168}
    there exists a quasi-universal family on $M_H \times S$. This
    family can be chosen to be universal due to the existence of $v'$.

    Let $E$ be a $H$-stable sheaf of class $v$ corresponding to a
    point of $M_H$. As $v = [E]$ is isotropic and $E$ is stable,
    $\dim_{\C}\Hom_S(E,E) =1$ and
    \[
      \dim_{\C}\Ext_S^1(E,E) = 1 + \dim_{\C}\Ext_S^2(E,E).
    \]
    As $E$ is slope stable and $\omega_S$ is a numerically trivial,
    $E \otimes \omega_S$ is slope stable of the same slope. Thus
    $\dim_{\C} \Hom_S(E,E \otimes \omega_S) = 0 \text{ or }1$ Then by
    Serre Duality and the equality above,
    $\dim_{\C} \Ext^1_S(E,E) \leq 2$.

    By construction, $M_H$ contains at least one closed point
    corresponding to a sheaf $F$ which is the pushforward of a
    semistable sheaf on the canonical cover. Thus $F$ is special by
    Proposition \ref{prop:derived-sheav-classes}, so
    $F \otimes \omega_S \iso F$ and $\dim_{\C}\Ext^2(E,E) = 1$. Hence
    $\dim_{\C}\Ext_S^1(F,F) = 2$. By Serre Duality and \cite[\S
    4.5]{MR2665168} $M_H$ is smooth at $F$ because the trace map on
    $\Ext_S^2(F,F)$ has zero kernel due to $F$ being special.

    As $M$ is smooth at $F$, $\dim M'_H = \dim_{\C}\Ext_S^1(F,F) = 2$
    for some connected component $M'_H$ of $M_H$. Hence
    $\dim_{\C}\Ext_S^1(E,E) \geq 2$ for all sheaves $E$ corresponding
    to points of $M'_H$. So $\dim_{\C}\Ext^1_S(E,E) = 2$ for all such
    $E$. Thus $M'_H$ is smooth of dimension $2$. Set $M = M'_H$. As
    $E$ is stable and has the same slope as $E \otimes \omega_S$, any
    map between them is an isomorphism. So $E$ is special as
    $\dim_{\C}\Hom(E,E \otimes \omega_S) = \dim_{\C}\Ext^2_S(E,E) =
    1$.

    Thus $M$ is a two dimensional, projective, smooth, fine moduli
    space of $H$-slope stable, special sheaves on $S$ of class $v$.

    By \cite[Corollary 2.8]{MR1827500} the universal sheaf
    $\mathcal{P}$ on $M \times S$ induces an  equivalence
    \[
      \Phi_{\mathcal{P}} \colon D(M) \to D(S).
    \]
    By Proposition \ref{prop:beilliptic-FM-partners}, $M$ is
    isomorphic to $S$. Thus the equivalence $\Phi_{\mathcal{P}}$
    induces an autoequivalence $\Phi$ of $D(S)$ after choosing an
    isomorphism $M \iso S$. By construction
    $[\Phi(\mathcal{O}_s)] = [\mathcal{P}_s] = v$.
\end{proof}
\end{proposition}

We now prove Theorem \ref{thm:main}.

\begin{proof}[Proof of Theorem \ref{thm:main}]
  First we describe the kernel of $\rho$. Let $\Psi \in \ker \rho$.
  As $\Psi$ is an integral transform, by a theorem of Orlov
  \cite[Theorem 5.14]{MR2244106}, $\Psi \iso \Psi_{\mathcal{P}}$ for
  some $\mathcal{P} \in D(S \times S)$.  As $S$ is a bielliptic
  surface, $\mathcal{P}$ is isomorphic to a shift of a sheaf
  \cite[Proposition 5.1]{MR3148638}. Thus
  $\Psi(\mathcal{O}_s) = \mathcal{P}_{\hat{s}}$ is a shift of a sheaf
  for any closed point $s \in S$. As $\Psi$ acts trivially on $N(S)$,
  $ch(\mathcal{P}_{\hat{s}}) = ch(\mathcal{O}_s) = (0,0,1)$.  Hence
  $\mathcal{P}_{\hat{s}}$ is a shift of a skyscraper sheaf for any
  closed point $s\in S$.  As $\Psi$ is a standard autoequivalence if
  and only if $\Psi(\mathcal{O}_s)$ is a shift of a skyscraper sheaf,
  $\Psi$ is a standard autoequivalence.

  The only standard autoequivalences that act trivially on $N(S)$ are
  $(\Aut S \ltimes \Pic^0 S) \times \mathbb{Z}[2]$. This is because
  the n-th power of the shift functor acts by $(-1)^n$ on
  $N(S)$. Tensoring by a line bundle $L$ act trivially on $N(S)$ if
  and only $c_1(L) = 0$, i.e. $L$ has degree zero. Automorphisms of
  $S$ act trivially on $N(S)$ because they preserve effective divisors
  and cannot exchange the fibres of the different elliptic fibrations
  as one has multiple fibres and the other does not.  \newline

  We now characterize the image of $\rho$. Let
  $\vp \in O_{\Delta}(N(S))$ and consider $v = \vp(0,0,1) \in
  \Delta$. Then $v \in \Delta$, $v^2 = 0$ and there exists
  $v' = \vp(1,0,0)$ such that $\langle v,v' \rangle = 1$. By
  Proposition \ref{prop:moduli} we can construct an autoequivalence
  $\Phi \in \Aut D(S)$ such that $\rho(\Phi)(0,0,1) = v$. Consider the
  isometry
  \[ 
    \vp' = (\rho(\Phi))^{-1} \circ \vp.
 \] 
 Then $\vp'(0,0,1) = (0,0,1)$. As $\vp'(1,0,0) = (1,D,s)$ is
 isotropic, $D^2 = 2s$. Thus $s = D^2/2$ and
 $\vp'(1,0,0) = (1,D,D^2/2)$ is the class of a line bundle $L$ with
 $c_1(L) =D$. Consider the isometry
\[
  \vp'' = \rho(L^* \otimes (-)) \circ \vp'.
\]
Notice that $\vp''$ acts by
\[ id_{H^0} \oplus \psi \oplus id_{H^4} \] on $N(S)$ where $\psi$ is
an isometry of $\num(S)$. Note that $\vp''$ respects the grading and
is an element of $O_{\Delta}(N(S))$ as it is a composite of elements
of $O_{\Delta}(N(S))$.  \newline

The group $\num(S)$ is isomorphic as a lattice to a single hyperbolic
plane $U$ with underlying group $\Z^2$ \cite[\S 1]{MR1038716}. The
group of isometries $O(U)$ is isomorphic to $\Z/2 \times \Z/2$. It is
generated by the involutions $\iota$, which acts by $-id$ on $U$, and
$\sigma$ which exchanges the two copies of $\Z$. Both of these give
rise to isometries of $N(S)$ by acting by the identity on $H^0(S,\Z)$
and $H^4(S,\Z)$ which we will denote by $\iota$ and $\sigma$ by an
abuse of notation.

Suppose the isometry $\iota$ is induced by an autoequivalence. As
$\iota$ fixes the class of a point and acts non-trivially on $N(S)$,
$\iota$ is induced by a standard autoequivalence which acts
non-trivially on $N(S)$. But standard autoequivalences which act
non-trivially on $N(S)$ act by tensoring by $\pm (1,D,D^2/2)$ for some
line bundle $L$ with $c_1(L) = D \ne 0$. However, $\iota$ does not
acts on $N(S)$ in this way as $\iota(1,0,0)=(1,0,0)$. Hence $\iota$ is
not induced by an autoequivalence. Similarly, $\sigma$ and
$\iota \circ \sigma$ are not induced by autoequivalences. Thus the
image of $\rho$ intersected with $O(\num(S))$ is trivial.

Note that $\iota$ preserves $\Delta$. However, $\sigma$ may not
preserve $\Delta$. The index of the image of $\rho$ will $2$ or $4$ in
$O_{\Delta}(N(S))$ depending on whether $\sigma$ preserves
$\Delta$. As $\sigma$ acts trivially on the two copies of $\Z$ in
$N(S)$ it is sufficient to study the action on $\num(S)$ by the
following Lemma.

\begin{lemma}
  A class $(r,D,s) \in \Delta$ if and only if $n \mid r$ and
  $(0,D,0) \in \Delta$. Thus
  $\Delta = n\Z \oplus \pi_*(\num(\widetilde{S})) \oplus \Z \subset \Z
  \oplus \num(S) \oplus \Z \iso N(S)$.
  \begin{proof}
    Suppose $n \mid r$ and $(0,D,0) \in \Delta$. Then
    $r = \tilde{r}n$ and there exists
    $\widetilde{D} \in \num(\widetilde{S})$ such that
    $\pi_!(0,\widetilde{D},0) = (0,\pi_*(\widetilde{D}),0) =
    (0,D,0)$. Then
      \[
        \pi_!(\tilde{r},\widetilde{D},s) = \pi_!(\tilde{r},0,0) +
        \pi_!(0,\widetilde{D},0) + \pi_!(0,0,s) = (r,D,s)
      \] 
      as $\pi_!(0,0,1) = (0,0,1)$.

      Suppose that $(r,D,s) \in \Delta$. Then there exists
      $[E] \in N(\widetilde{S})$ such that
      $\pi_!([\widetilde{E}]) = (r,D,s)$. By computing the Mukai
      pairing of $(r,D,s)$ with the classes $(1,0,0)$ and $(0,0,1)$ we
      see that $ch_2([E])= s$ and $r = n \rk(E)$. So
      $(r,0,0), (0,0,s) \in \Delta$ as
      $\pi_!(\rk(E)[\mathcal{O}_{\widetilde{S}}]) = (r,0,0)$ and
      $\pi_!(s[\mathcal{O}_{\tilde{s}}]) = (0,0,s)$. Then
      \[
        (r,D,s) - (0,0,s) - (r,0,0) = (0,D,0) \in \Delta.
      \]
  \end{proof}
\end{lemma}

If $(r,D,s) \in \Delta$ then
$\sigma(r,D,s)= (r,\sigma(D),s) \in \Delta$ if and only if
$(0,\sigma(D),0) \in \Delta$. To determine whether $\sigma$ preserves
$\Delta$ we reduce to studying classes of the form $(0,D,0)$. By abuse
of notation, we will denote the class $(0,D,0) \in N(S)$ by $D$ and we
write $D \in \Delta$ for $(0,D,0) \in \Delta$.

\begin{lemma}\label{lem:classes_in_Delta}
  The classes $A,B \in \Delta$ but $\frac{1}{k}A \notin \Delta$. If
  $S$ is non-split, then $\frac{1}{k}B \notin \Delta$.
\begin{proof}
  The classes $A, B \in \Delta$ as $\pi_*(\tilde{A}) = A$ and
  $\pi_*(\tilde{B}) = B$.

  Suppose that $\frac{1}{n}A \in \Delta$. Then there exist
  $0 \not\equiv \widetilde{D} \in \num(\widetilde{S})$ such that
  $\pi_*(\widetilde{D}) = \frac{1}{n}A$. As
  $\widetilde{D} \cdot \widetilde{D} = n(\pi_*D,\pi_*D) =
  n(\frac{1}{n}A,\frac{1}{n}A) = 0$, by \cite[Proposition
  2.3]{MR1285957}, $D \equiv mE$ for some $0 \ne m \in \Z$ and $E$ an
  elliptic curve. Then by the push-pull formula we have
  \[
    0 = A \cdot \pi_*(mE) = \pi_*(\pi^*A \cdot mE).
  \]
  As the pushforward of points is injective on cohomology, we have
  \[
    0 = \pi^*A \cdot mE = n\widetilde{A} \cdot mE = nm (\widetilde{A} \cdot E).
  \]
  So $\widetilde{A} \cdot E = 0$. As $E$ and $\widetilde{A}$ are
  irreducible curves, by \cite[Proposition 2.1]{MR1285957}
  $E =T_{\tilde{s}}(\widetilde{A})$, so $E \equiv \widetilde{A}$. But
  $\pi_*(mE) = \pi_*(m\widetilde{A}) = mA \not\equiv \frac{1}{k}A$,
  which is a contradiction. Hence $\frac{1}{k}A \notin \Delta$.

  A similar argument holds for $\frac{1}{k}B$ when $S$ is a non-split
  bielliptic by replacing $\widetilde{A}$ by $\widetilde{B}$.
\end{proof}
\end{lemma}

Note that $\sigma$ interchanges the generators of $\num(S)$. We will
consider separate cases to determine the index of the image of $\rho$.

We will use the following repeatedly: A class $D \in \Delta$ if
and only if $D' = D + (aA + bB) \in \Delta$ with $a,b \in \Z$. Clearly
if $D \in \Delta$ then $D' \in \Delta$. Conversely, if
$D' \in \Delta$, then $D = D' - (aA +bB) \in \Delta$ as $\Delta$ is a
subgroup.

\begin{description}
\item[Split Bielliptic] Suppose that $S$ is a split bielliptic
  surface. Then $\sigma$ interchanges $\frac{1}{n}A$ and $B$. But by
  the above claim $\frac{1}{n}A \notin \Delta$ but $B \in \Delta$, so
  $\sigma$ does not preserve $\Delta$. Hence the index is $2$.
\item[Bielliptic of type 2] By Lemma \ref{lem:classes_in_Delta} we have
  $\frac{1}{2}A, \frac{1}{2}B \notin \Delta$ and $A,B \in \Delta$.
  Consider $D = \frac{a}{2}A + \frac{b}{2}B$ with $a,b \in \Z$. Then
  $\sigma(D) = \frac{b}{2}A + \frac{a}{2}B$. By adding or subtracting
  multiples of $A$ and $B$ we can reduce to the cases when
  $a,b \in \{0,1\}$. We have 3 cases:

  \begin{enumerate}
  \item If $a = b = 0$ then $D \in \Delta$ and $\sigma(D) \in \Delta$.

  \item Suppose $a = 0$ and $b = 1$. Then
    $\sigma(D) = \frac{1}{2}A \not\in \Delta$ and
    $D = \frac{1}{2}B \not\in \Delta$. A similar argument show that
    $D,\sigma(D) \not\in \Delta$ for $a = 1$ and $b = 0$.

  \item Suppose that $a = b = 1$. Then
    $D = \frac{1}{2}A + \frac{1}{2}B = \sigma(D)$. Hence
    $D \in \Delta$ if and only if $\sigma(D) \in \Delta$.
  \end{enumerate}

  Thus $\sigma$ preserves $\Delta$ and the index is $4$.
\item[Bielliptic of type 4] By Lemma \ref{lem:classes_in_Delta} 
  we have $\frac{1}{3}A, \frac{1}{3}B \notin \Delta$ and
  $A,B \in \Delta$.  Consider $D = \frac{a}{3}A + \frac{b}{3}B$ with
  $a,b \in \Z$ and $\sigma(D) = \frac{b}{3}A + \frac{a}{3}B$. By
  adding or subtracting multiples of $A$ and $B$ we can reduce to the
  cases when $a,b \in \{0,1,-1\}$. We have 4 cases:

  \begin{enumerate}
  \item If $a = b =0$. Then $D \in \Delta$ and $\sigma(D) \in \Delta$.

  \item Suppose that $a = b = 1$ Then
    $\sigma(D) = \frac{1}{3}A + \frac{1}{3}B = D$. Hence
    $D \in \Delta$ if and only if $\sigma(D) \in \Delta$. A similar
    argument works for $a = b = -1$.

  \item Suppose that $m = a$ and $b =1 $. Then
    $D = \frac{1}{3}B \not\in \Delta$ and
    $\sigma(D) = \frac{1}{3}A \not\in \Delta$. Similarly for
    $a=0,b=-1$ and $a=1,-1, b=0 $ we have $D \not\in \Delta$ and
    $\sigma(D) \not\in \Delta$.

  \item Suppose that $a= 1$ and $b=-1$. Then
    $\sigma(D) = -\frac{1}{3}A+ \frac{1}{3}B = -D$.  As $\Delta$ is a
    subgroup $-D \in \Delta$ if and only if $D \in \Delta$. Hence
    $D \in \Delta$ if and only if $\sigma(D) \in \Delta$. A similar
    argument works for $a = -1$ and $b = 1$.
  \end{enumerate}

  Thus $\sigma$ preserves $\Delta$ and the index is $4$.
\item[Bielliptic of type 6] Note that $\frac{1}{2}A \notin \Delta$ by
  a similar argument to Lemma \ref{lem:classes_in_Delta}. Then as
  $\sigma$ interchanges $\frac{1}{2}A$ and $2(\frac{1}{2}B) = B$,
  $\sigma$ does not preserve $\Delta$. Hence the index is $2$.
\end{description}
\end{proof}

\section{Relative Fourier-Mukai Transforms and bielliptic surfaces}
\label{sec:relat-four-mukai-bielliptic}

For a bielliptic surface $S$, relative Fourier-Mukai transforms with
respect to either elliptic fibration $p_A$ or $p_B$ give rise to
autoequivalences of $D(S)$ in the following way.

\begin{proposition}\label{prop:rel-fm-bielliptic}
  Let $S$ be a bielliptic surface and $p_A \colon S \to A/G$ and
  $p_B \colon S \to B/G$ its two relatively minimal elliptic
  fibrations. Then a relative Fourier-Mukai transform with respect to
  either fibration induces an autoequivalence on $D(S)$ which is
  non-standard.
  \begin{proof}
    Let $\Phi_{Rel} \colon D(Y) \to D(S)$ be a relative Fourier-Mukai
    transform induced by one of the two fibrations. By Proposition
    \ref{prop:beilliptic-FM-partners}, $Y$ is isomorphic to $S$. After
    choosing an isomorphism $g \colon Y \to S$, the composite
    $ \Psi = \Phi_{Rel} \circ g^*$ is an autoequivalence of $D(S)$. It
    is non-standard because $ch(\Psi(\mathcal{O}_s)) = (0,af,b)$ where
    $f$ is the fibre of the elliptic fibration.
  \end{proof}
\end{proposition}

\begin{example}\label{ex:relat-four-mukai-examples}
  Note that for either fibration $p_A$ or $p_B$ of $S$ we have an
  autoequivalence corresponding to the matrix
\[ P = 
  \begin{pmatrix}
    1 & 1\\
     0 & 1\\
  \end{pmatrix}
\]
given by Theorem \ref{thm:relative-fm}. We have an autoequivalence
$\Psi_B$, constructed by composing the relative Fourier-Mukai
transform along $p_A$ associated to $P$ and tensoring by a suitable
line bundle, which acts on $N(S)$ by
\begin{align*}
  (1,0,0) &\mapsto (1,0,0)\\
  (0,0,1) &\mapsto (0,B,1)\\
  (0,B',0) &\mapsto (0,B',0)\\
  (0,A',0) &\mapsto (\lambda_{p_A},A',0).
\end{align*}
Note $\Psi_B$ sends $(0,A,0)$ to $(n,A,0)$.
\newline

Suppose that $S$ is split. Then the fibration $p_A \colon S \to A/G$
admits a section, i.e. $\lambda_{p_A} = 1$. Then there is a relative
Fourier-Mukai functor $\hat{\Psi}$ that corresponds to the matrix
\[
  \begin{pmatrix}
  0 & 1 \\
  -1 & 0 \\
  \end{pmatrix}
\]
given by Theorem \ref{thm:relative-fm} which acts on $N(S)$ by
\begin{align*}
  (1,0,0) &\mapsto (0,(-1/n)A,0)\\
  (0,0,1) &\mapsto (0,B,0)\\
  (0,B,0) &\mapsto (0,0,1)\\
  (0,(1/n)A,0) &\mapsto (1,0,0).
\end{align*}
\end{example}

We now prove Theorem \ref{thm:generators}.

\begin{proof}[Proof of Theorem \ref{thm:generators}]
  As $S$ is split, $k=1$ and $|G|= n= \deg \pi$ and
  $\widetilde{S} \iso A \times B$. Let $\Phi \in \Aut D(S)$. Consider
  $v = \rho(\Phi)(0,0,1)$. Then $v \in \Delta$, $v^2= 0$ and there
  exists $v' = \rho(\Phi)(1,0,0)$ such that
  $\langle v, v' \rangle = 1$.

  We will  construct an autoequivalence $\Psi \in \Aut D(S)$
  which is the composite of standard autoequivalences  and relative
  Fourier-Mukai transforms along $p_A$ and $p_B$ such that
  $\rho(\Psi)(0,0,1) = v$.

  We separate the argument into three cases:

  \begin{enumerate}
  \item Suppose that $v = \pm(0,0,1)$. Then $\Psi = id$ or
    $[1]$.\newline

  \item Suppose that $v = (0,D,s)$. As $\langle v,v \rangle = 0$,
    $D = aA$ or $b B$ for $a,b \in \Z$,
    $a, b \ne 0$. Suppose that $D = a A$. As there
    exists $v'=\vp(1,0,0) = (r', (a'/n) A + b' B,s')$
    such that $\langle v,v'\rangle = 1$, we have
    \[
      a (B \cdot A) b' - sr' =1.
    \]
    As $\lambda_{p_B} = B \cdot A$, $\gcd(a \lambda_{p_B},s)
    =1$. Therefore there exists a relative Fourier-Mukai transform,
    $\hat{\Phi}$, along $p_B$ such that $\rho(\hat{\Phi})$ sends
    $(0,0,1)$ to $v =(0,a A,s)$. Then set $\Psi = \hat{\Phi}$. A
    similar argument for $D= b B$ will work to construct a relative
    Fourier-Mukai transform along $p_A$ which sends $(0,0,1)$ to
    $(0,b B,s)$.  \newline

  \item Suppose that $v = (r,a A + b B,s)$ with $r \ne 0$. We can
    assume that $r>0$ after applying $\rho([1])$. Then $r = nc$ with
    $c \in \mathbb{N}$, as $v \in \Delta$. As $v^2 = 0$ we have
    \[
      v = \left(nc, a A + b B, ab/c \right).
    \]
    Note one of $a,b$ is non zero as otherwise $v$ would be
    divisible.  \newline

    Suppose $a = 0$, so $v = (nc, b B,0)$.  Then we can
    apply the relative Fourier-Mukai transform $\hat{\Psi}$ which
    sends
    \[
      (nc ,b B,0) \mapsto (0,-c A,b)
    \]
    and reduce to case $(2)$. \newline

    Suppose that $a \ne 0$. After tensoring by $A$ we can assume
    $a >0$. Let $gcd(c,a)=d$ for some
    $d \in \mathbb{N}$. We can write
    $c  = d c'$ and $a = d a'$
    with $gcd(a',c') =1$. Thus $v$ has the form
    \[
      v = (n d c', da' A + b B,
      a' b /c').
    \]
    We have two operations given by $\rho (\-- \otimes (-1/n)A)$ and
    $\rho(\Psi_B^{-1})$ which act on $ndc'$ and
    $da'$ in the following way:
    \begin{align*} 
      \rho(\-- \otimes (-1/n)A):& (n dc', da') \mapsto 
                                  (n dc',d(a' - c'))\\ 
      \rho(\Psi_B^{-1}):& (n dc', d a') \mapsto (n d(c' -a'), d a').
    \end{align*} 
    This is just the Euclidean algorithm on $c' $ and $a'$. Thus we
    can reduce $a'$ to $1$ and $c'$ to $0$ and proceed as in $(2)$.
  \end{enumerate}

  Consider the autoequivalence $\Psi^{-1} \circ \Phi$ whose image
  under $\rho$ sends $(0,0,1)$ to $(0,0,1)$. So $\Psi^{-1} \circ \Phi$
  is a standard autoequivalence.  Thus we can express $\Phi$ as a
  composite of standard autoequivalences and relative Fourier-Mukai
  transforms.
\end{proof}

\bibliographystyle{plain}
\bibliography{bielliptic_bib}
\end{document}